\documentclass[preprint,11pt]{elsarticle}

\usepackage{amssymb,amsthm,amsmath,hyperref,txfonts}

\usepackage{graphicx,float}
\usepackage{color}

\newcommand \nc{\newcommand}
\newtheorem{theorem}{Theorem}[section]
\newtheorem{lemma}[theorem]{Lemma}
\newtheorem{proposition}[theorem]{Proposition}
\newtheorem{corollary}[theorem]{Corollary}
\newtheorem{definition}[theorem]{Definition}

\newtheorem{remark}[theorem]{Remark}

\newtheorem*{assw}{Assumption (ICKHw)}
\newtheorem*{assr}{Assumption (RICKH)}
\newtheorem*{assrw}{Assumption (RICKHw)}

\nc{\ba}{\begin{array}}\nc{\ea}{\end{array}}
\nc{\be}{\begin{eqnarray}}\nc{\ee}{\end{eqnarray}}
\nc{\beq}{\begin{equation}}\nc{\eeq}{\end{equation}}
\nc{\bex}{\begin{eqnarray*}}\nc{\eex}{\end{eqnarray*}}
\nc{\btm}{\begin{theorem}} \nc{\etm}{\end{theorem}}
\nc{\blm}{\begin{lemma}} \nc{\elm}{\end{lemma}}
\nc{\R}{\mathbb{R}} \nc{\va}{\varepsilon} \nc{\ls}{\limits}

\allowdisplaybreaks

\begin{document}

\begin{frontmatter}



\title{Inviscid limit of the inhomogeneous incompressible Navier-Stokes equations under the weak Kolmogorov hypothesis
 in $\mathbb{R}^3$}


\author[W]{Dixi Wang}
\address[W]{Department of Mathematics, University of Florida, Gainesville, FL 32611, United States of America}
\ead[W]{dixiwang@ufl.edu}

\author[Y]{Cheng Yu}
\address[Y]{Department of Mathematics, University of Florida, Gainesville, FL 32611, United States of America}
\ead[Y]{chengyu@ufl.edu}

\author[Z]{Xinhua Zhao\corref{cor1}}
\cortext[cor1]{Corresponding author}
\address[Z]{School of Mathematics, South China University of Technology, Guangzhou 510641, China}
\ead[Z]{xinhuazhao211@163.com}

\begin{abstract} 
In this paper,
we consider the inviscid limit of inhomogeneous incompressible Navier-Stokes equations  under the weak Kolmogorov hypothesis in $\R^3$. In particular, we first deduce the Kolmogorov-type hypothesis in $\mathbb{R}^3$, which yields the uniform bounds of $\alpha^{th}$-order fractional derivatives of $\sqrt{\rho^\mu}{\bf u}^\mu $ in $L^2_x$ for some $\alpha>0$, independent of the viscosity. The uniform bounds can provide  strong convergence of $\sqrt{\rho^{\mu}}\bf u^{\mu}$ in $L^2$ space. This  shows that  the inviscid limit  is a weak solution to the corresponding Euler equations.
\end{abstract}

\vspace{4mm}


\end{frontmatter}


\tableofcontents

\vspace{4mm}
\section {Introduction}
\setcounter{equation}{0}
The main purpose of this paper is to study the vanishing viscosity limit of inhomogeneous Navier-Stokes equations in $\mathbb{R}^3$ under the well known hypothesis by Kolmogorov \cite
{Kolmogorov2, Kolmogorov1}. In particular, a weaker version Assumption (KHw) of Assumption (KH) which was derived in \cite{Chen-Glimm2012, Chen-Glimm2019} in a periodic domain. This  can provide  the convergence of weak solutions of the Navier-Stokes equations  through a subsequence to a solution of the Euler equations. 
We are particularly interested in extending these results to the inhomogeneous fluids in the whole space, because of  the special structure of inhomogeneous equations and different formulation of  Kolmogorov hypothesis in $\mathbb{R}^3 $.

We are particularly interested in the following Navier-Stokes equations for the inhomogeneous fluids:
\begin{align}\label{maineq}
\left\{
\begin{array}{l}
 \rho_{t} + \mathrm{div}(\rho {\bf u}) = 0,\\
(\rho {\bf u})_t + \mathrm{div}(\rho {\bf u}\otimes {\bf u}) + \nabla{p} - \mu\Delta{\bf u}= \rho \mathbf{f} ,\\
\mathrm{div}~{\bf u}=0,
\end{array}
\right.
\end{align}
with initial condition

\begin{equation}\label{I.C.}
(\rho,\rho {\bf u})|_{t=0}=( \rho_0,{\bf m}_0 )({\bf x}),~~~ \bf x \in \mathbb{R}^3,
\end{equation}
and boundary condition
\begin{equation}\label{B.C.}
{\bf u} \rightarrow {\bf u}_\infty,~~~as~|{\bf x}| \rightarrow +\infty,~~for~all~t \geq 0,
\end{equation}
where $ {\bf u}_\infty $ is fixed vector in $ \mathbb{R}^3. $ 
Here
 $\rho, {\bf u}, P$ stand for the density, velocity and pressure of the fluid respectively, $\mu>0$ is the viscosity coefficient,
and ${\bf
f}={\bf f}({\bf x}, t)=(f_1, f_2, f_3)({\bf x},t)$ denotes a given external force.

 There are many literatures studying the existence and uniqueness of solutions to inhomogeneous incompressible Navier-Stokes equations, see Refs. \cite{AKM1990,AVK1974,Simon1990} and the references therein. Lady\v{z}enskaja and Solonnikov \cite{Ladysolo1975} first addressed the question of unique solvability of \eqref{maineq}. The global existence of weak solutions to \eqref{maineq} with the large initial data was first proved by Lions \cite{Lions1}. For clarity of presentation, we assume that  ${\bf u}_\infty=0$ in this paper. More precisely, for any $ T > 0$ and any $ {\bf f} \in L^2(0,T;L^2(\mathbb{R}^3))$, and the initial data satisfies the following conditions:
\begin{align}\label{ics}
  & \rho_0 \geq0~~a.e.~in~\mathbb{R}^3,~~\rho_0 \in L^\infty(\mathbb{R}^3), \nonumber \\
  &{\bf m}_0 \in L^2(\mathbb{R}^3),~~{\bf m}_0=0~~a.e.~on~\left\{\rho_0=0 \right\},
  ~~\frac{{\bf m}_0^2}{\rho_0} \in L^1(\mathbb{R}^3),
  \\
  &
  (1/{\sqrt{\rho_0}}) 1_{(\rho_0 < \delta_0)} \in L^2(\mathbb{R}^3)~~for~some~positive~ constant~\delta_0.
  \end{align}
If the initial data verifies the above conditions,
Lions \cite{Lions1} proved the global existence of weak solutions in the following sense:
\begin{definition}[\cite{Lions1}]\label{defi-1}
 ~For any $  T > 0, $ $ (\rho^\mu,{\bf u}^\mu)(t,x) $ is a weak solution on $ [0,T] $
 of \eqref{maineq}-\eqref{B.C.} if
 \begin{itemize}
\item
$ \sqrt{\rho^\mu}{\bf u}^\mu \in L^\infty(0,T;L^2(\mathbb{R}^3)),~
  \nabla {\bf u}^\mu \in L^2((0,T)\times\mathbb{R}^3),~
    \rho^\mu \in C([0,T];L^p(\mathbb{R}^3))~~for~1 \leq p< \infty, $

\item
 For any $ \boldsymbol{\phi} \in C_0^\infty(\mathbb{R}_+ \times \mathbb{R}^3;\mathbb{R}^3) $ such that $ \mathrm{div}~\boldsymbol{\phi}=0,$
 \begin{align}\label{momentum for NS}
 & \int _0^T\int _{\mathbb{R}^3} \left(\rho^{\mu}{\bf u}^{\mu} \cdot\boldsymbol{\phi}_t+(\rho^{\mu} {\bf u}^{\mu} \otimes {\bf u}^{\mu}): \nabla\boldsymbol{\phi}+\rho^{\mu} \mathbf{f} \cdot\boldsymbol{\phi}\right)\, \mathrm{d}{\bf x}\, \mathrm{d}t
  +\int_{\mathbb{R}^3}{\bf m}_0({\bf x})\cdot\boldsymbol{\phi}(0,{\bf x})\,\mathrm{d}{\bf x}
\nonumber \\
 =&\int _0^T\int _{\mathbb{R}^3} \mu\nabla {\bf u}^\mu\cdot\nabla \boldsymbol{\phi}\, \mathrm{d}{\bf x}\, \mathrm{d}t.
\end{align}
\item For any $ \phi \in C_0^\infty(\mathbb{R}_+ \times \mathbb{R}^3;\mathbb{R})$ 
\begin{equation}
  \int _0^T\int _{\mathbb{R}^3} (\rho^{\mu} \phi_t+{\rho^{\mu}} {\bf u}^{\mu} \cdot \nabla \phi)\, \mathrm{d}{\bf x}\, \mathrm{d}t+ \int_{\mathbb{R}^3} \rho^{\mu}_0 \phi(0,{\bf x})\, \mathrm{d}{\bf x}=0.
\end{equation}
\item The energy inequality holds for any $ t \in [0,T] $
 \begin{equation}\label{energy}
    \int _{\mathbb{R}^3} \frac{\rho^\mu |{\bf u^\mu}|^2}{2}\, \mathrm{d}{\bf x}
    +\int _0^t\int _{\mathbb{R}^3} \mu|\nabla {\bf u^\mu}|^2\, \mathrm{d}{\bf x}\, \mathrm{d} s \leq \int _{\mathbb{R}^3} \frac{|{\bf m}_0|^2}{2 \rho_0}\, \mathrm{d}{\bf x} +\int _0^t\int _{\mathbb{R}^3} \rho^\mu{\bf u^\mu} \cdot {\bf f}\, \mathrm{d}{\bf x}\, \mathrm{d} s.
  \end{equation}
 \end{itemize}
\end{definition}
\noindent However the uniqueness of the weak solutions in three-dimensional space is still an interesting open question, which is linked to the global regularity.

The inviscid limit for Navier-Stokes equations has also been extensively studied (see Refs.\cite{Bernicot2016,Chemin1996,Msmoudi2007} for instance). 
The goal of this paper is to show that the weak solutions of \eqref{maineq}-\eqref{B.C.} converge to a weak solution of the following Euler equations, if they are under the  Kolmogorov hypothesis:
\begin{align}\label{Eulereq}
\left\{
\begin{array}{l}
 \rho_{t} + \mathrm{div}(\rho {\bf u}) = 0,\\
(\rho {\bf u})_t + \mathrm{div}(\rho {\bf u}\otimes {\bf u}) + \nabla{P} =\rho \mathbf{f} ,\\
\mathrm{div}~{\bf u}=0,
\end{array}
\right.
\end{align}
with the same initial data \eqref{I.C.}. The definition of global weak solution to \eqref{Eulereq} and \eqref{I.C.} is given as follows:
\begin{definition}\label{EWD}
For any $  T > 0, $  $ (\rho,{\bf u})(t,x) $ is called a global weak solution on $ [0,T] $
 of \eqref{Eulereq} with initial data $ (\rho_0, {\bf m}_0) $ if $ (\rho,{\bf u}) $ satisfies:
 \begin{itemize}
\item
  $ (\rho,{\bf u}) $ solves the system \eqref{Eulereq} in the sense of distributions in $ [0,T]\times \mathbb{R}^3$ , i.e.\\
(a). For any $ \boldsymbol{\phi} \in C_0^\infty(\mathbb{R}_+ \times \mathbb{R}^3;\mathbb{R}^3) $ such that $ \mathrm{div}~\boldsymbol{\phi}=0,$
 \begin{align}\label{momentum}
  \int _0^T\int _{\mathbb{R}^3} (\sqrt{\rho}\sqrt{\rho} {\bf u} \cdot\boldsymbol{\phi}_t+(\sqrt{\rho} {\bf u} \otimes \sqrt{\rho}{\bf u}): \nabla\boldsymbol{\phi}+\rho \mathbf{f} \cdot\boldsymbol{\phi})\, \mathrm{d}{\bf x}\, \mathrm{d}t +\int_{\mathbb{R}^3}{\bf m}_0({\bf x})\cdot\boldsymbol{\phi}(0,{\bf x})\,\mathrm{d}{\bf x}
   =0,
\end{align}
(b). For any $ \phi \in C_0^\infty(\mathbb{R}_+ \times \mathbb{R}^3;\mathbb{R}) $
\begin{equation}
  \int _0^T\int _{\mathbb{R}^3} (\rho \phi_t+\sqrt{\rho}\sqrt{\rho} {\bf u} \cdot \nabla \phi)\, \mathrm{d}{\bf x}\, \mathrm{d}t+ \int_{\mathbb{R}^3} \rho_0 \phi(0,{\bf x})\, \mathrm{d}{\bf x}=0.
\end{equation}
\item \eqref{I.C.} holds in $ \mathcal{D}'(\mathbb{R}^3), $
\item the energy inequality holds for any $ t \in (0,\infty) $
 \begin{equation}
    \int _{\mathbb{R}^3} \frac{\rho |\bf u|^2}{2}\, \mathrm{d}{\bf x}
   \leq \int _{\mathbb{R}^3} \frac{|{\bf m}_0|^2}{2 \rho_0}\, \mathrm{d}{\bf x} +\int _0^t\int _{\mathbb{R}^3} \rho{\bf u} \cdot {\bf f}\, \mathrm{d}{\bf x}\, \mathrm{d} t.
  \end{equation}
\end{itemize}
\end{definition}
In this paper, we aim to investigate such limits of global weak solutions from the inhomogeneous incompressible Navier-Stokes equations to the corresponding Euler equations under
the weaker  Kolmogorov-type hypothesis, which was particularly motivated by \cite{Chen-Glimm2012} for incompressible fluids and \cite{Chen-Glimm2019} for compressible fluids. 
Compared to these two models, the special features of inhomogeneous incompressible Navier-Stokes equations bring new difficulties to the mathematical analysis. Specifically, on one hand, the pressure-density relation for the compressible flows does not hold for the inhomogeneous flow anymore. On the other hand,
the term $\partial_t(\rho \bf u)$ is nonlinear in $(\rho,\bf u)$ thus requires additional argument to handle the regularity in time, which is not necessary for
 homogeneous incompressible Euler system.

 In fact, the same issue arises in the mathematical study of the energy equality for the weak solutions of compressible or inhomogeneous flows, see reference \cite{CY,FGW, LS,Yu} and their references.
  To circumvent this time regularity assumption,
  one approach is to formulate the equations in terms of the density and energy equality for the momentum $\bf m=\rho \bf u$
and obtain the energy equality by multiplying the momentum equation by $(\rho\bf u)^{\varepsilon}/\rho^{\varepsilon}$ instead of $\bf u^{\varepsilon}$, where $\varepsilon$ is a suitable regulation. However, $(\rho\bf u)^{\varepsilon}/\rho^{\varepsilon}$  cannot keep the divergence free structure, thus requires a commutator estimate involving the pressure and additional regularity of the pressure is needed. In \cite{CY, Yu}, the authors use
 $(\varphi(t) \bf u^{\varepsilon})^{\varepsilon}$ where $\varphi(t)$ is sufficient nice function. This function is divergence free so it does not require any additional regularity of the pressure. It succeeded in 
 removing the difficulty related to the regularity of pressure in these works.
  Such ideas carried into this current paper as well.

 In addition to the special feature of inhomogeneous incompressible Navier-Stokes equations, we are mainly interested in the inviscid limit problem in the whole space.
 We adopt the Fourier integrals to express  the Kolmogorov-type hypothesis in $\mathbb{R}^3 $, which is different from
the case in $ \mathbb{T}_P=[-\frac{P}{2},\frac{P}{2}]^3 , P>0 ,$ in \cite{Chen-Glimm2012,Chen-Glimm2019}.  Our goal is to show that the same conclusion  in \cite{Chen-Glimm2012,Chen-Glimm2019} holds in $\R^3$.
 To this end, we introduce the Kolmogorov-type hypothesis for inhomogeneous incompressible fluids in $\mathbb{R}^3$  using physical notion in \cite{McComb}. The details can be founded in section \ref{section2}.
Accordingly, it is interesting to investigate the inviscid limit  under the weaker version of  Kolmogorov-type hypothesis in $\R^3$.
We present our main result as follows.

\begin{theorem}\label{thm-R1} If the weak solutions $ (\rho^\mu, {\bf u^\mu}) $ of \eqref{maineq}-\eqref{B.C.} as in Definition \ref{defi-1} are under Assumption (RICKHw) \eqref{rass-1}, there exists a subsequence (still denoted) $ (\rho^\mu, {\bf u^\mu}) $ and a function $ (\rho, {\bf u}) $ such that as $ \mu \rightarrow 0,$
\begin{equation}\label{conv}
 \rho^\mu \rightarrow \rho ~~~weakly~in~L^p((0,T)\times \mathbb{R}^3),~
 \sqrt{\rho^\mu}{\bf u^\mu} \rightarrow \sqrt{\rho}{\bf u}~in~L^2((0,T)\times \mathbb{R}^3),
\end{equation}
 where $ 1 < p<\infty $, and $ (\rho, {\bf u}) $ is a weak solution of \eqref{Eulereq} with initial data $ (\rho_0,{\bf m}_0) $.
\end{theorem}

\begin{remark}\label{rem:3}
The inviscid limit for compressible Navier-Stokes equations in $\mathbb{R}^3$ for any $\gamma>\frac{3}{2}$ can be obtained by the same method as this paper. However, for homogeneous incompressible Navier-Stokes equations, the total energy $ \mathcal{E}(t) $ per unit mass vanishes as defined in Kolmogorov hypothesis in \cite{McComb}.
\end{remark}

Meanwhile, we can establish a similar result in $ \mathbb{T}_P=[-\frac{P}{2},\frac{P}{2}]^3 , P>0.$ The proof will be given in Appendix at the end of this paper.
\begin{theorem}\label{rem:2}
If we consider the fluid in a domain with period $\mathbb{T}_P=[-P/2,P/2]^3 \subset \mathbb{R}^3, P>0,$ one can deduce the same result as Theorem \ref{thm-R1}. However, in this case, the Kolmogorov hypothesis is deduced as follows:
 \begin{equation}\label{w}
   \sup_{k \geq k_{\ast}} \Big(|\mathbf{k}|^{3+ \beta}\int_0^T|\widehat{\sqrt{\rho}
  {\bf u}}(t,\mathbf{k})|^2\mathrm{d}t \Big)\leq C_T,~~for~some~\beta>0   
  \end{equation}
different from Assumption (RICKHw). In fact, for the domain $\mathbb{T}_P $, the total energy $ \mathcal{E}(t) $ per unit mass satisfies
\begin{equation}\label{Fourier}
  \mathcal{E}(t)=\frac{1}{\int_{\mathbb{T}_P} \rho^\mu \mathrm{d}{\bf x} }\int _{\mathbb{T}_P} \frac{\rho^\mu |{\bf u^\mu}|^2}{2}\, \mathrm{d}{\bf x}=\sum_{k \geq 0}E(t,k)= \sum_{k \geq 0}4 \pi q(t,k)k^2,
\end{equation}
and the weighted velocity $\sqrt{\rho}{\bf u}(t,{\bf x}) $ can be expanded to Fourier series, i.e.,
\begin{equation}\label{RParseval}
  \sqrt{\rho}{\bf u}(t,{\bf x})=\sum_{\bf k} \widehat{\sqrt{\rho}{\bf u}}(t,{\bf k})e^{i {\bf k} \cdot {\bf x}}.
\end{equation}
The proof highly relies on \eqref{w} and \eqref{RParseval}, see Appendix \ref{sec:appendix}.
\end{theorem}

In general, the global existence theory of weak solutions to Euler equations in three dimensional space has not been established yet.
We make our attempt to find a way of obtaining existence result of the inhomogeneous Euler equations under special conditions. In general, the vanishing viscosity limit of Navier-Stokes equations is not a solution to the corresponding Euler equations, since the uniform bounds cannot guarantee the convergence of the nonlinear term $\rho^{\mu}\bf u^{\mu}\otimes\bf u^{\mu}$. 

  In order to pass the limit of the nonlinear term $ \rho^\mu {\bf u}^\mu\otimes  {\bf u}^\mu,$  if we follow the techniques in   \cite{Chen-Glimm2019} to show strong convergene of $\rho^{\mu}\bf u^{\mu}$, the additional regularity on the pressure is required. It seems not easy to obtain such a regularity for the weak solutions. However, we do not have a pressure law for inhomogeneous fluids, and the velocity is not good enough to be a test function even after smoothing. 
 Our key idea is to show the strong convergence of $\sqrt{\rho^{\mu}}\bf u^{\mu}$ in $L^2_{t,x}$, this yields the weak limit of a subsequence is a weak solution to the Euler equations. We emphasize the fact that our argument does not need any additional regularity on the pressure, since we find a suitable divergence free test function. This yields our Lemma \ref{lem-R4} so as to show the strong convergence of $\sqrt{\rho^{\mu}}\bf u^{\mu}.$

The manuscript is organized as follows. In section 2, we present the  Kolmogorov hypotheses for the inhomogeneous incompressible fluids in the whole space, and  a sufficient weaker version of Assumption (RICKHw) for this current project. In section 3, we derive the compactness of weak solutions of the Navier-Stokes equations when the viscosity coefficient vanishes, which is crucial to obtain the weak solution to the Euler equations. In Section 4, we prove our main result based on the compactness results in Section 3. We give an outline of the proof to Theorem 1.5  in Section 5 as an appendix.

\section{The Kolmogorov hypotheses for the inhomogeneous incompressible fluids}\label{section2}
In this section, we introduce the Kolmogorov hypotheses for the inhomogeneous incompressible fluids in $\mathbb{R}^3$ and the corresponding Kolmogorov-type hypothesis (RICKH) in mathematical terms. Note that,
two fundamental assumptions for the isotropic incompressible turbulence were proposed by  Kolmogorov  \cite{Kolmogorov2,Kolmogorov1}:
\begin{itemize}
\item[(i)]~At sufficiently high wavenumbers, the energy spectrum $ E(t,k) $ can depend only on the fluid viscosity $ \mu, $ the dissipation rate $ \varepsilon $ and the wavenumber $ k $ itself.\\
\item[(ii)]~$ E(t,k) $ should be independent of the viscosity as the Reynolds number tends to infinity:
\begin{equation*}
E(t,k) \approx \alpha \varepsilon^{2/3}k^{-3/5}
\end{equation*}
in the limit of infinite Reynolds number, where $ \alpha $ may depend on t, but is independent of $ k, \varepsilon. $
\end{itemize}

Under  the above Kolmogorov's two hypotheses, Chen-Glimm  \cite{Chen-Glimm2012,Chen-Glimm2019} interpreted in mathematical terms for the incompressible and compressible Kolmogorov-type hypothesis in $\mathbb{T}_P$.
We can generalize them to $\R^n$ for the fluid equations as follows.
 \begin{assr}\label{ickh}
 For any $ T >0, $ there exists $ C_T>0$ and $ k_{*} $(sufficiently large) depending on $ \rho_0, {\bf m}_0 $ and $ {\bf f} $ but independent of the viscosity $ \mu $ such that, for $ k=|{\bf k}|\geq k_* ,$
 \begin{equation}
  \int_0^T E(t,{\bf k}) \mathrm{d}t \leq C_T k^{-\frac{5}{3}}.
 \end{equation}
 \end{assr}
On  the one hand, McComb \cite{McComb} defined the spectral tensor $ Q_{\alpha \beta}(t,{\bf k}) $, and stated the relationship between the trace of $ Q_{\alpha \beta}(t,{\bf k})  $ and the energy  $ \mathcal{E}(t) $ per unit mass of fluid at time t, i.e.
\begin{equation}\label{trace}
  2\mathcal{E}(t)=trQ_{\alpha \beta}(t,{\bf r})|_{r=0}=tr \int _0^\infty q(t,k)k^2 \mathrm{d}k \int D_{\alpha \beta}({\bf k})\mathrm{d} \Omega_{\bf k},
\end{equation}
where $D_{\alpha \beta}({\bf k}) $, $q(t,k)  $ and $\mathrm{d} \Omega_{\bf k}  $ denote projection operator, spectral density and the elementary solid angle in wavenumber space respectively.
Using \eqref{trace}, Leslie \cite{LS1983} obtained the following crucial result
\begin{align}\label{E}
  \mathcal{E}(t)=&\frac{4 \pi}{3}tr \delta_{\alpha \beta} \int_0^\infty q(t,k)k^2 \mathrm{d}k \nonumber \\ 
   =&\int_0^\infty 4 \pi k^2 q(t,k) \mathrm{d}k \nonumber \\
   =&\int_0^\infty E(t,k) \mathrm{d}k.
 \end{align}

On the other hand, from the energy inequality \eqref{energy} of the weak solutions $ (\rho^\mu, {\bf u^\mu}) $, the Gronwall inequality yields to 
  \begin{equation}\label{Energy}
    \int _{\mathbb{R}^3} \frac{\rho^\mu |u^\mu|^2}{2}\, \mathrm{d}{\bf x}
    +\int _0^t\int _{\mathbb{R}^3} \mu|\nabla {\bf u^\mu}|^2\, \mathrm{d}{\bf x}\, \mathrm{d} t \leq M_T,
  \end{equation}
where $ M_T $ is a positive constant that depends on intial data, {\bf f} and T, but independent of $ \mu $. By \eqref{E} and \eqref{Energy},  the total energy $ \mathcal{E}(t) $ per unit mass at time t for the inhomogeneous turbulence is:
\begin{equation}
  \mathcal{E}(t)=\frac{1}{\int _{\mathbb{R}^3} \rho^\mu\mathrm{d}x}\int _{\mathbb{R}^3} \frac{\rho^\mu |{\bf u^\mu}|^2}{2}\, \mathrm{d}{\bf x}=\int _0^\infty E(t,k)\mathrm{d}k= \int _0^\infty 4 \pi q(t,k)k^2\mathrm{d}k.
\end{equation}

Furthermore, when we consider the domain $\mathbb{T}_P=[-P/2,P/2]^3 \subset \mathbb{R}^3, P>0,$ the wavevector $ {\bf k}=(k_1,k_2,k_3)=\frac{2 \pi}{P}(n_1,n_2,n_3) \in \mathbb{R}^3, $ with $ n_j=0, \pm1, \pm2,\cdots $ and $ j=1,2,3, $ is discrete.
When the setting is in  $ \mathbb{R}^3 $, in some sense, we can view that it is the large box limit as $ P \rightarrow \infty $. Since the wavevector is continuous, 
 we introduce the  Fourier transform for  the weighted velocity $\sqrt{\rho}{\bf u}(t,{\bf x}) $ in the {\bf x}-variable as follows
\begin{equation}
 \widehat{\sqrt{\rho} {\bf u}}(t,\mathbf{k}) = \frac{1}{(2 \pi)^3}\int _{\mathbb{R}^3} \sqrt{\rho} {\bf u}(t,{\bf x})e^{-i{\bf k} \cdot {\bf x}}\, \mathrm{d}{\bf x},
\end{equation}
thus
\begin{equation}
  \sqrt{\rho} {\bf u}(t,\mathbf{x})=\int _{\mathbb{R}^3} \widehat{\sqrt{\rho} {\bf u}}(t,{\bf k})
  e^{i{\bf k} \cdot {\bf x}}\, \mathrm{d}{\bf k}.
\end{equation}
By Parseval identity, we have
\begin{equation}\label{Parseval}
 \|\sqrt{\rho^\mu}u^\mu\|_{L^2_{\bf x}}^2=\|\widehat{\sqrt{\rho^\mu}u^\mu}\|_{L^2_{\bf k}}^2 .
\end{equation}
Clearly, it is usually more convenient to use the spherical  coordinates $ (k,\theta, \varphi)$, where $$ k=|{\bf k}|, \,0 \leq \theta \leq 2\pi, 0 \leq \varphi \leq  \pi.$$ Note that the facts  $$\widehat{\sqrt{\rho^\mu}{\bf u}^\mu}(t,{\bf k})=\widehat{\sqrt{\rho^\mu}{\bf u}^\mu}(t,k,\theta,\varphi),$$and \eqref{Parseval}, we derive that
\begin{align}\label{spherical}
 \int _0^\infty E(t,k)~\mathrm{d}k=&\frac{1}{\int _{\mathbb{R}^3} \rho^\mu\mathrm{d}x}\int _{\mathbb{R}^3} \frac{\rho^\mu |{\bf u^\mu}|^2}{2}\, \mathrm{d}{\bf x} \nonumber \\
 =&\frac{1}{2\int _{\mathbb{R}^3} \rho_0\mathrm{d}x}\int _{\mathbb{R}^3} |\widehat{\sqrt{\rho^\mu} {\bf u^\mu}}(t,{\bf k})|^2\, \mathrm{d}{\bf k} \nonumber \\ 
 =&\frac{1}{2\int _{\mathbb{R}^3} \rho_0\mathrm{d}x}\int _0^{\pi} \int _0^{2 \pi} \int _0^\infty |\widehat{\sqrt{\rho^\mu}{\bf u}^\mu(t,k,\theta,\varphi)}|^2~ k^2 \sin \varphi\, \mathrm{d}k~\mathrm{d}\theta~\mathrm{d}\varphi,
\end{align}
where we used  $$\int_{\R^3}\rho^{\mu}({\bf x},t)~\mathrm{d}{\bf x} = \int_{\R^3}\rho_0({\bf x})~\mathrm{d}{\bf x}, $$  which can be derived from the mass equation in \eqref{maineq}.

Thanks to {\sl Assumption (RICKH)}, $E(t,{\bf k})=E(t,k,\theta,\varphi) $ and the equality \eqref{spherical}, we are able to obtain the following weaker version of {\sl Assumption (RICKHw)}.  This can provide the uniform bound for our compactness analysis in this current paper.

 \begin{assrw}\label{rickhw}
For any $ T>0, $ there exists $ C_T>0 $ and $ k_{\ast} $(sufficiently large)  depending on $ \rho_0,{\bf m}_0 $ and $ \mathbf{f} $ but independent of the viscosity $\mu $ such that, for $ k=|\mathbf{k}|\geq k_ {\ast},$
\begin{equation}\label{rass-1}
	\sup_{k \geq k_{\ast}} \Big(|\mathbf{k}|^{3+ \beta}\int_0^T\int _0^{\pi} \int _0^{2 \pi} |\widehat{\sqrt{\rho^\mu}
	{\bf u}^\mu}(t,k,\theta,\varphi)|^2~ k^2 \sin \varphi\,\mathrm{d}\theta~\mathrm{d}\varphi~\mathrm{d}t \Big)\leq C_T,~~for~some~\beta>0.
\end{equation}
 \end{assrw}
\section{Compactness of weak solutions }\label{section3}

We develop the compactness of weak solutions when the viscosity coefficient vanishes in this section.
The following lemma is crucial to show the compactness of weak solutions.
\begin{lemma}\label{lem-R1}Under Assumption (RICKHw), for any $ T \in (0, +\infty), $ there exists $C>0,$
  independent of $ \mu>0, $ such that
 ~ \begin{equation}
 	\int _0^T\int _{\mathbb{R}^3} |D_x^{\alpha}( \sqrt{\rho^\mu} {\bf u}^\mu)|^2\, \mathrm{d}{\bf x} \, \mathrm{d} t \leq C,
\end{equation}
where $ \alpha \in(0, 1+\frac{\beta}{2}), $
$ C$ the generic positive constants depending on
$ \rho_0, {\bf m}_0, \mathbf{f}, \alpha, k_ \ast$ and $ T>0, $ but independent of $ \mu>0. $
 \end{lemma}
 
 \begin{remark}
  A similar version for periodic domain case were given in \cite{Chen-Glimm2012, Chen-Glimm2019}, where the proof relied on the definition of fractional derivatives via Fourier series.
 \end{remark}
\begin{proof} Note that the definition of fractional derivatives via Fourier transform, the Parseval identity,  the spherical  coordinates  and {\sl Assumption (RICKHw)}, we have
\begin{align}
	\int _0^T\int _{\mathbb{R}^3} |D_x^{\alpha}( \sqrt{\rho^\mu} {\bf u}^\mu)(t,{\bf x})|^2\, \mathrm{d}{\bf x}\, \mathrm{d} t
	=&\int _0^T \int_{\mathbb{R}^3} |\widehat{ D_x^\alpha( \sqrt{\rho^\mu} {\bf u}^\mu)}(t,{\bf k})|^2\, \mathrm{d}{\bf k}\mathrm{d}t \nonumber \\
	=& \int _0^T \int_{\mathbb{R}^3}|{\bf k}|^{2 \alpha} |\widehat{ \sqrt{\rho^\mu} {\bf u}^\mu}(t,{\bf k})|^2\, \mathrm{d}{\bf k}\mathrm{d}t \nonumber \\
	=& \int _0^T \int _0^{\pi} \int _0^{2 \pi} \int _0^\infty k^{2 \alpha} |\widehat{ \sqrt{\rho^\mu} {\bf u}^\mu}(t, k, \theta,\varphi)|^2k^2 \sin \varphi\, \mathrm{d}k~\mathrm{d}\theta~\mathrm{d}\varphi ~\mathrm{d}t\nonumber \\
	=& \int _0^T \int _0^{\pi} \int _0^{2 \pi}\int _0^ {k_{\ast}} k^{2 \alpha}|\widehat{\sqrt{\rho^\mu}{\bf u}^\mu}|^2 k^2 \sin \varphi\, \mathrm{d}k~\mathrm{d}\theta~\mathrm{d}\varphi ~\mathrm{d}t\nonumber \\
	&+\int _0^T \int _0^{\pi} \int _0^{2 \pi}\int _{k_{\ast}}^\infty k^{2 \alpha}|\widehat{\sqrt{\rho^\mu}{\bf u}^\mu}|^2 k^2 \sin \varphi\, \mathrm{d}k~\mathrm{d}\theta~\mathrm{d}\varphi ~\mathrm{d}t\nonumber \\
	 \leq&	C k_ \ast^{2 \alpha}\int _0^T\int _{\mathbb{R}^3} |\sqrt{\rho^\mu} {\bf u}^\mu|^2\, \mathrm{d}{\bf x}\, \mathrm{d}t+ C\int _{k_{\ast}}^\infty k^{2 \alpha-3- \beta} \, \mathrm{d}k \nonumber \\
	 \leq& C,
\end{align}	
where $ \alpha<1+\frac{\beta}{2}. $
\end{proof}

With Lemma \ref{lem-R1} at hand, we are able to have the following uniform bound of $ \sqrt{\rho^\mu}{\bf u}^\mu $ in $ \mu. $
\begin{corollary}\label{cor-2}
Under {\sl Assumption (RICKHw)}, for any $ T >0, $ there exists $C>0,$ independent of $ \mu>0, $ such that
for $ q=q(\beta)>2, $
~\begin{equation}
	\|\sqrt{\rho^\mu}{\bf u}^\mu\|_{ L^q((0,T)\times \mathbb{R}^3)}\leq C.
\end{equation}
\end{corollary}
\begin{proof} In view of Gagliardo-Nirenberg interpolation inequality and Lemma \ref{lem-R1}, we have
\begin{equation}
	\|\sqrt{\rho^\mu}{\bf u}^\mu\|_{L^q}
\leq \|\sqrt{\rho^\mu}{\bf u}^\mu\|_{L^2}^{\alpha_1}
\| \sqrt{\rho^\mu} {\bf u}^\mu\|_{H^ \alpha}^{1-\alpha_1},
\end{equation}
where $ \frac{1}{q}=\frac{\alpha_1}{2}+ \big( \frac{1}{2}-\frac{\alpha}{3} \big)(1- \alpha_1)$ for $ 0< \alpha_1<1 .$
Note that  $\sqrt{\rho^\mu}{\bf u}^\mu $is uniformly bounded in $ L^2(0,T;H^\alpha(\mathbb{R}^3)), $ we get $ q=2+\frac{4 \alpha}{3}>2. $
\end{proof}

Now it turns to show that $\sqrt{\rho^{\mu}}\bf u^{\mu}$ is equicontinuous with respect to the space variable $x$, which is necessary for our argument to get its convergence.
\begin{lemma}\label{lem-R3}Under {\sl Assumption (RICKHw)}, for any $ T \in (0, +\infty), $
  $ \sqrt{\rho^\mu} {\bf u}^\mu $ is equicontinuous with respect to the space variable ${\bf x}$ in $ L^2((0,T)\times \mathbb{R}^3) $, independent of $ \mu, $ i.e.,
 ~\begin{equation}
  \int _0^T\int _{\mathbb{R}^3} | \sqrt{\rho^\mu}{\bf u}^\mu(t,{\bf x}+ \Delta {\bf x})-\sqrt{\rho^\mu} {\bf u}^\mu(t,{\bf x})|^2\, \mathrm{d}{\bf x}\, \mathrm{d}t
  \rightarrow 0,~~as ~ \Delta {\bf x} \rightarrow 0.
 \end{equation}
 \end{lemma}
 \begin{proof}Using Lemma \ref{lem-R1} and Parseval identity, we obtain
 \begin{align*}
   &\int _0^T\int _{\mathbb{R}^3} | \sqrt{\rho^\mu}{\bf u}^\mu(t,{\bf x}+ \Delta {\bf x})-\sqrt{\rho^\mu} {\bf u}^\mu(t,{\bf x})|^2\, \mathrm{d}{\bf x}\, \mathrm{d}t \nonumber \\ 
   =& \int_0^T \int_{\mathbb{R}^3} \Big| \int_{\mathbb{R}^3}\widehat{\sqrt{\rho^\mu}{\bf u}^\mu}(t,{\bf k})e^{i{\bf k} \cdot {\bf x}}(e^{i{\bf k} \cdot \Delta {\bf x}}-1)\,\mathrm{d}{\bf k}\Big|^2\, \mathrm{d}{\bf x}\,\mathrm{d}{t} \nonumber \\
   \leq& C\int_0^T \int_{\mathbb{R}^3} \int_{\mathbb{R}^3} e^{2i{\bf k} \cdot {\bf x}}\, \mathrm{d}{\bf k}
    \int_{\mathbb{R}^3} |\widehat{\sqrt{\rho^\mu}{\bf u}^\mu}(t,{\bf k})|^2(e^{i{\bf k} \cdot \Delta {\bf x}}-1)^2\, \mathrm{d}{\bf k}\, \mathrm{d}{\bf x}\,\mathrm{d}{t} \nonumber \\
    \leq& C |\Delta {\bf x}|^{2 \alpha}\int_{\mathbb{R}^3} \delta(2{\bf x})\, \mathrm{d}{\bf x}
    \int_0^T \int _{\mathbb{R}^3} |{\bf k}|^{2 \alpha}
    |\widehat{ \sqrt{\rho^\mu}{\bf u}^\mu}(t,{\bf k})|^2 \, \mathrm{d}{\bf k}\, \mathrm{d}t 
    \nonumber \\
   \leq& C |\Delta {\bf x}|^{2 \alpha} \int _0^T\int _{\mathbb{R}^3} |D_{\bf x}^\alpha(\sqrt{\rho^\mu}{\bf u}^\mu)|^2\, \mathrm{d}{\bf x}\, \mathrm{d}t  \nonumber \\
    \leq& C|\Delta {\bf x}|^{2 \alpha} ,
 \end{align*}
 where $ \delta(\bf x) $ is the Dirac delta function.
 \end{proof}

 To show the strong convergence of $\sqrt{\rho^{\mu}}{\bf u}^{\mu}$ in $L^2$ space, it is crucial to have the following lemma. 
\begin{lemma}\label{lem-R4}Under {\sl Assumption (RICKHw)}, for any $ T >0 $,
  $ \sqrt{\rho^\mu} {\bf u}^\mu $ is equicontinuous with respect to the time variable t in $ L^2((0,T- \Delta t )\times \mathbb{R}^3) $, independent of $ \mu,$ i.e.
 ~\begin{equation}
  \int _0^{T-\Delta t}\int _{\mathbb{R}^3}
  | \sqrt{\rho^\mu}{\bf u}^\mu(t+ \Delta t,{\bf x})
  -\sqrt{\rho^\mu} {\bf u}^\mu(t,{\bf x})|^2\, \mathrm{d}{\bf x}\, \mathrm{d}t
  \rightarrow 0,~~as ~ \Delta t \rightarrow 0.
 \end{equation}
 \end{lemma}
 \begin{proof}
For simplicity, we drop the superscript $ \mu $ of
 $ \sqrt{\rho^\mu}{\bf u}^\mu $ in the following proof. For any $ \varphi(t) \in \mathcal{D}(0,+ \infty), $ we have
 \begin{align}\label{rt}
 &\int _0^{T- \Delta t}\int _{\mathbb{R}^3}
  | \sqrt{\rho}{\bf u}(t+ \Delta t,{\bf x})
  -\sqrt{\rho} {\bf u}(t,{\bf x})|^2\, \mathrm{d}{\bf x}\, \mathrm{d}t \nonumber \\
   =&\int _0^{T- \Delta t}\int _{\mathbb{R}^3} [\rho {\bf u}(t+\Delta t,{\bf x})
  -\rho {\bf u}(t,{\bf x})][{\bf u}(t+\Delta t,{\bf x})-{\bf u}(t,{\bf x})]\, \mathrm{d}{\bf x}\, \mathrm{d}t \nonumber \\
   &+\int _0^{T- \Delta t}\int _{\mathbb{R}^3}
  \sqrt{\rho} {\bf u}(t+\Delta t,{\bf x}){\bf u}(t,{\bf x})
  [ \sqrt{\rho}(t+\Delta t,{\bf x})-\sqrt{\rho}(t,{\bf x})]\, \mathrm{d}{\bf x}\, \mathrm{d}t \nonumber \\
   &+\int _0^{T- \Delta t}\int _{\mathbb{R}^3}
  \sqrt{\rho }{\bf u}(t,{\bf x}){\bf u}(t+\Delta t,{\bf x})
  [ \sqrt{\rho}(t,{\bf x})-\sqrt{\rho}(t+\Delta t,{\bf x})]\, \mathrm{d}{\bf x}\, \mathrm{d}t \nonumber \\
   =& \int _0^{T- \Delta t}\int _{\mathbb{R}^3} [\rho {\bf u}(t+\Delta t,{\bf x})
  -\rho {\bf u}(t,{\bf x})]\left\{ [{\bf u}(t+\Delta t,{\bf x})
  - \varphi(t){\bf u}^\epsilon(t+\Delta t,{\bf x})]
  -[{\bf u}(t,{\bf x})-\varphi(t){\bf u}^ \epsilon(t,{\bf x})]\right\}\, \mathrm{d}{\bf x}\, \mathrm{d}t \nonumber \\
  &+\int _0^{T- \Delta t}\int _{\mathbb{R}^3} [\rho {\bf u}(t+\Delta t,{\bf x})
  -\rho {\bf u}(t,{\bf x})] [\varphi(t){\bf u}^\epsilon(t+
  \Delta t,{\bf x})-\varphi(t){\bf u}^ \epsilon(t,{\bf x})]\, \mathrm{d}{\bf x}\, \mathrm{d}t \nonumber \\
   &+\int _0^{T- \Delta t}\int _{\mathbb{R}^3}
  \sqrt{\rho} {\bf u}(t+\Delta t,{\bf x}){\bf u}(t,{\bf x})
  [ \sqrt{\rho}(t+\Delta t,{\bf x})-\sqrt{\rho}(t,{\bf x})]\, \mathrm{d}{\bf x}\, \mathrm{d}t \nonumber \\
   &+\int _0^{T- \Delta t}\int _{\mathbb{R}^3}
  \sqrt{\rho} {\bf u}(t,{\bf x}){\bf u}(t+\Delta t,{\bf x})
  [ \sqrt{\rho}(t,{\bf x})-\sqrt{\rho}(t+\Delta t,{\bf x})]\, \mathrm{d}{\bf x}\, \mathrm{d}t \nonumber \\
  =&I_1+I_2+I_3+I_4,
 \end{align}
 where
 \begin{equation*}
  {\bf u}^\epsilon(t,{\bf x})=(j_{\epsilon} \ast {\bf u})(t,{\bf x})
  = \int_{|t-s|\leq \epsilon}\int_{|{\bf x}-{\bf y}|\leq \epsilon} j_{\epsilon}(t-s,{\bf x}-{\bf y}) {\bf u}(s,{\bf y})\, \mathrm{d}{\bf y}\, \mathrm{d}s,
  \end{equation*}
   and $j_\epsilon$ is a standard mollifier. Note that $\varphi(t){\bf u^\epsilon}(t+\Delta t,{\bf x}) $, $ \varphi(t){\bf u^\epsilon}(t,{\bf x}) $ are well defined in $ \mathcal{D}(0,+\infty). $

We divide the equality \eqref{rt} into four terms, and each term will be considered separately.
First, we can rewrite $ I_3 $ as following
\begin{align}\label{I3}
 I_3=& \int _0^{T- \Delta t}\int _{\mathbb{R}^3}
  \sqrt{\rho} {\bf u}(t+\Delta t,{\bf x}){\bf u}(t,{\bf x})
  [ (\sqrt{\rho}(t+\Delta t,{\bf x})-\sqrt{\rho^\epsilon}(t+\Delta t,{\bf x}))-(\sqrt{\rho}(t,{\bf x})-\sqrt{\rho^\epsilon}(t,{\bf x}))]\, \mathrm{d}{\bf x}\, \mathrm{d}t \nonumber \\
  &+\int _0^{T- \Delta t}\int _{\mathbb{R}^3}
  \sqrt{\rho} {\bf u}(t+\Delta t,{\bf x}){\bf u}(t,{\bf x})
  [ \sqrt{\rho^\epsilon}(t+\Delta t,{\bf x})-\sqrt{\rho^\epsilon}(t,{\bf x})]\, \mathrm{d}{\bf x}\, \mathrm{d}t \nonumber \\
    =& \int _0^{T- \Delta t}\int _{\mathbb{R}^3}\sqrt{\rho} ({\bf u}(t+\Delta t,{\bf x}){\bf u}(t,{\bf x})
  (\sqrt{\rho}(t+\Delta t,{\bf x})-\sqrt{\rho^\epsilon}(t+\Delta t,{\bf x}))\, \mathrm{d}{\bf x}\, \mathrm{d}t \nonumber \\
  &+\int _0^{T- \Delta t}\int _{\mathbb{R}^3}\sqrt{\rho} {\bf u}(t+\Delta t,{\bf x}){\bf u}(t,{\bf x})(\sqrt{\rho^\epsilon}(t,{\bf x})-\sqrt{\rho}(t,{\bf x}))\, \mathrm{d}{\bf x}\, \mathrm{d}t \nonumber \\
  &+\int _0^{T- \Delta t}\int _{\mathbb{R}^3}\sqrt{\rho} {\bf u}(t+\Delta t,{\bf x})
  {\bf u}(t,{\bf x}) ( \sqrt{\rho^\epsilon}(t+\Delta t,{\bf x})-\sqrt{\rho^\epsilon}(t,{\bf x}))\, \mathrm{d}{\bf x}\, \mathrm{d}t \nonumber \\ 
   =&I^1_3+I^2_3+I^3_3. 
\end{align}
Here we decompose $ {\bf u}$ into $ {\bf u}= {\bf u} 1_{ \left\{ \rho<\delta_0 \right\}}
+{\bf u} 1_{ \left\{ \rho \geq\delta_0 \right\}}={\bf u}_1+{\bf u}_2$, then in view of $\frac{1}{\sqrt{\rho_0}}1_{\left\{ \rho_0< \delta_0 \right\}} \in L^2(\mathbb{R}^3),$ we can get
\begin{equation}
\|{\bf u} 1 _{\left\{ \rho <\delta_0 \right\}}\|_{L^l}
\leq \big\|\frac{1}{\sqrt{\rho}}1_{\left\{ \rho <\delta_0 \right\}}
\big\|_{L^2}\|\sqrt{\rho}{\bf u}\|_{L^q},~~l=\frac{2q}{q+2}.
\end{equation}
Thus ${\bf u}_1 \in L^{l}((0,T)\times\mathbb{R}^3)$.
As for ${\bf u}_2$, since $\sqrt{\rho}{\bf u} \in L^\infty(0,T;L^2(\mathbb{R}^3))$, we have 
${\bf u}_2 \in L^\infty(0,T;L^2(\mathbb{R}^3)).$
Moreover, it holds that as $ \epsilon \rightarrow 0,$
\begin{equation}\label{u-conv}  
\|{\bf u}_1-{\bf u}_1^\epsilon\|_{L^{l}(0,T;L^{l}(\mathbb{R}^3))} \rightarrow 0,
~~\|{\bf u}_2-{\bf u}_2^\epsilon\|_{L^\infty(0,T;L^2(\mathbb{R}^3))} \rightarrow 0. 
\end{equation}
For $ I_3^1+I_3^2, $ using 
$\sqrt{\rho^\epsilon}\rightarrow \sqrt{\rho} \,~
\mathrm{in}\, C([0,T], L^p(\mathbb{R}^3)),~2 \leq p<\infty$, as $\epsilon=\epsilon(\Delta t) \rightarrow 0 $, we have
\begin{align}
 I_3^1+I_3^2 \leq& \int _0^{T- \Delta t}\int _{\mathbb{R}^3} \sqrt{\rho} {\bf u}~
  {\bf u}_1(\sqrt{\rho^\epsilon}-\sqrt{\rho})
  \, \mathrm{d}{\bf x}\, \mathrm{d}t \nonumber \\ 
  &+\int _0^{T- \Delta t}\int _{\mathbb{R}^3} \sqrt{\rho} {\bf u}~
  {\bf u}_2(\sqrt{\rho^\epsilon}-\sqrt{\rho})
  \, \mathrm{d}{\bf x}\, \mathrm{d}t\nonumber \\  
 \leq& C\|\sqrt{\rho}{\bf u}\|_{L^q(0,T;L^q(\mathbb{R}^3))} \|{\bf u}_1\|_{L^l(0,T;L^l(\mathbb{R}^3))}\|\sqrt{\rho^\epsilon}-\sqrt{\rho}\|_{L^\infty(0,T;L^{p_1}(\mathbb{R}^3))} \nonumber \\[1mm]
  &+C\|\sqrt{\rho}{\bf u}\|_{L^q(0,T;L^q(\mathbb{R}^3))} \|{\bf u}_2\|_{L^\infty(0,T;L^2(\mathbb{R}^3))}\|\sqrt{\rho^\epsilon}-\sqrt{\rho}\|_{L^\infty(0,T;L^{p_2}(\mathbb{R}^3))} \nonumber \\ 
  & \rightarrow 0,~~as~\Delta t \rightarrow 0,
\end{align}
where 
$ q>4 $, $ p_1 $ and $ p_2 $ satisfy $ \frac{1}{q}+\frac{1}{l}+\frac{1}{p_1}=1 $ and $\frac{1}{q}+\frac{1}{2}+\frac{1}{p_2}=1 $ respectively.\\
For $ I_3^3, $ using H\"{o}lder inequality, we obtain
\begin{align}
 I_3^3 \leq& C\|\sqrt{\rho}{\bf u}\|_{L^{q}(0,T;L^{q}(\mathbb{R}^3))}
 \|{\bf u}_1\|_{L^{l}(0,T;L^{l}(\mathbb{R}^3)} \|\sqrt{\rho^\epsilon}(t+\Delta t,{\bf x})-\sqrt{\rho^\epsilon}(t,{\bf x})\|_{L^\frac{2q}{q-4}([0,T] \times \mathbb{R}^3)}  \nonumber \\ 
  &+C\|\sqrt{\rho}{\bf u}\|_{L^\infty(0,T;L^2(\mathbb{R}^3))}
  \|{\bf u}_2\|_{L^\infty(0,T;L^2(\mathbb{R}^3)}) \|\sqrt{\rho^\epsilon}(t+\Delta t,{\bf x})-\sqrt{\rho^\epsilon}(t,{\bf x})\|_{L^\infty([0,T] \times \mathbb{R}^3)} \nonumber \\ 
   &\rightarrow 0,~~as~\Delta t \rightarrow 0,
\end{align}
where $ q>4. $

Similarly, we have as $\Delta t \rightarrow0, I_4 \rightarrow0, $ independent of $ \mu. $

For $ I_2, $ from $ \mathrm{div}~{\bf u}^\epsilon=0 $ and Definition \ref{defi-1} , we get
\begin{align}\label{i2}
 I_2=&\int _0^{T- \Delta t}\int _{\mathbb{R}^3}\int_t^{t+\Delta t}
 (\rho {\bf u})_s(s,{\bf x})\cdot [\varphi(t){\bf u}^\epsilon(t+\Delta t,{\bf x})-\varphi(t){\bf u}^ \epsilon(t,{\bf x})]\, \mathrm{d}s\, \mathrm{d}{\bf x}\,\mathrm{d}t \nonumber \\
=& \int _0^{T- \Delta t}\int _{\mathbb{R}^3}\int_t^{t+\Delta t}
(\rho {\bf u} \otimes {\bf u})(s,{\bf x}): \nabla [\varphi(t){\bf u}^\epsilon(t+\Delta t,{\bf x})-\varphi(t){\bf u}^ \epsilon(t,{\bf x})]\,
\mathrm{d}s\, \mathrm{d}{\bf x}\, \mathrm{d}t \nonumber \\
&-\int _0^{T- \Delta t}\int _{\mathbb{R}^3}\int_t^{t+\Delta t}
\mu \nabla {\bf u}(s,{\bf x}) \cdot \nabla[\varphi(t){\bf u}^\epsilon(t+\Delta t,{\bf x})-\varphi(t){\bf u}^ \epsilon(t,{\bf x})]\,
\mathrm{d}s\, \mathrm{d}{\bf x}\, \mathrm{d}t \nonumber \\
&+\int _0^{T- \Delta t}\int _{\mathbb{R}^3}\int_t^{t+\Delta t}
 (\rho \mathbf{f})(s,{\bf x})\cdot [\varphi(t){\bf u}^\epsilon(t+\Delta t,{\bf x})-\varphi(t){\bf u}^ \epsilon(t,{\bf x})]\, \mathrm{d}s\, \mathrm{d}{\bf x}\,
\mathrm{d}t \nonumber \\
 =& I_2^1+I_2^2+I_2^3.
\end{align}
Notice that, for any $x\in\R^3$,
\begin{align}\label{1}
 | \varphi(t){\bf u}^\epsilon(t,{\bf x}) |
 \leq&\frac{C}{\epsilon^4} \Big|\int_{|t-s|\leq \epsilon}\int_{|{\bf x}-{\bf y}|\leq \epsilon}
 j(\frac{t-s}{\epsilon},\frac{{\bf x}-{\bf y}}{\epsilon}){\bf u}(s,{\bf y})\, \mathrm{d}{\bf y}\, \mathrm{d}s\Big| \nonumber \\
  \leq& \frac{C}{\epsilon^{4}}\|{\bf u}_1\|_{L^{l}(0,T;L^{l}(\mathbb{R}^3)}\Big( \int_{|\tau|\leq 1}\int_{|{\bf z}|\leq 1}
 |j(\tau,{\bf z})|^{l'}\epsilon^4\, \mathrm{d}{\bf z}\, \mathrm{d}\tau \Big)^{1/{l'}}  \nonumber \\
 &+\frac{C}{\epsilon^4} \|{\bf u}_2\|_{L^\infty(0,T;L^2(\mathbb{R}^3)}
  \Big( \int_{|\tau|\leq 1}\int_{|{\bf z}|\leq 1}
 |j(\tau,{\bf z})|^{2}\epsilon^4\, \mathrm{d}{\bf z}\, \mathrm{d}\tau \Big)^{1/2} \nonumber \\ 
 \leq& \frac{C}{\epsilon^{2(q+2)/q}},
\end{align}
and
\begin{align}\label{2}
| \varphi(t) \nabla{\bf u}^\epsilon(t,{\bf x}) |
 \leq& \frac{C}{\epsilon^5} \Big|\int_{|t-s|\leq \epsilon}\int_{|{\bf x}-{\bf y}|\leq \epsilon}
 \partial_2 j(\frac{t-s}{\epsilon},\frac{{\bf x}-{\bf y}}{\epsilon}) {\bf u}(s,{\bf y})\, \mathrm{d}{\bf y}\, \mathrm{d}s\Big| \nonumber \\
\leq& \frac{C}{\epsilon^{5}}\|{\bf u}_1\|_{L^{l}(0,T;L^{l}(\mathbb{R}^3)}\Big( \int_{|\tau|\leq 1}\int_{|{\bf z}|\leq 1}
 |\partial_z j(\tau,{\bf z})|^{l'}\epsilon^4\, \mathrm{d}{\bf z}\, \mathrm{d}\tau \Big)^{1/{l'}}  \nonumber \\
 &+\frac{C}{\epsilon^5} \|{\bf u}_2\|_{L^\infty(0,T;L^2(\mathbb{R}^3)}
  \Big( \int_{|\tau|\leq 1}\int_{|{\bf z}|\leq 1}
 |\partial_z j(\tau,{\bf z})|^{2}\epsilon^4\, \mathrm{d}{\bf z}\, \mathrm{d}\tau \Big)^{1/2} \nonumber \\ 
 \leq& \frac{C}{\epsilon^{(3q+4)/q}}.
\end{align}
For $ I_2^1+I_2^3, $ by \eqref{1},\eqref{2} and H\"{o}lder inequality, we have
\begin{align}\label{213}
 I_2^1+I_2^3 \leq&C \Delta t~\|\sqrt{\rho} {\bf u}\|_{L^2}^2|\varphi \nabla{\bf u}^\epsilon| +C \Delta t~\|\rho\|_{L^2}\|\mathbf{f}\|_{L^2}|\varphi{\bf u}^\epsilon | \nonumber \\
  \leq& \frac{C \Delta t}{\epsilon^{(3q+4)/q}}.
\end{align}
For $ I_2^2, $ in view of H\"{o}lder inequality and the energy inequality in Definition \ref{defi-1}, we obtain
\begin{align}\label{22}
 I_2^2 \leq& C\Big|\int _0^{T- \Delta t}\int _{\mathbb{R}^3}\int_t^{t+\Delta t}
\sqrt{\mu} \nabla {\bf u}(s,{\bf x})\,
\mathrm{d}s \cdot [\sqrt{\mu}\nabla{\bf u}^\epsilon(t+\Delta t,{\bf x})-\sqrt{\mu}\nabla{\bf u}^ \epsilon(t,{\bf x})]\, \mathrm{d}{\bf x}\, \mathrm{d}t\Big| \nonumber \\
\leq& C\Big[\int _0^{T- \Delta t}\int _{\mathbb{R}^3}\Big(\int_t^{t+\Delta t}\sqrt{\mu} \nabla {\bf u}(s,{\bf x})\,
\mathrm{d}s\Big)^2 \, \mathrm{d}{\bf x}\, \mathrm{d}t\Big]^{\frac{1}{2}}
\Big[\int _0^{T- \Delta t}\int _{\mathbb{R}^3}\mu  |\nabla{\bf u}^\epsilon|^2\, \mathrm{d}{\bf x}\, \mathrm{d}t\Big]^{\frac{1}{2}} \nonumber \\
 \leq&C (\Delta t)^{\frac{1}{2}}\Big[\int _0^{T- \Delta t}\int _{\mathbb{R}^3}\int_t^{t+\Delta t}\mu |\nabla {\bf u}|^2(s,{\bf x})\,
\mathrm{d}s\, \mathrm{d}{\bf x}\, \mathrm{d}t\Big]^{\frac{1}{2}} \nonumber \\
 \leq& C \Delta t.
\end{align}
Plugging \eqref{213} and \eqref{22} into \eqref{i2} yields
\begin{equation}
  I_2 \leq \frac{C \Delta t}{\epsilon^{{(3q+4)/q}}}+C \Delta t \leq \frac{C \Delta t}{\epsilon^{(3q+4)/q}}.
 \end{equation}
By choosing $ \epsilon=(C\Delta t)^{\frac{q}{6q+8}}, $ we deduce that
\begin{equation}\label{I2}
  I_2\leq \Delta t^\frac{1}{2}  \rightarrow 0,~as ~\Delta t \rightarrow 0.
\end{equation}

To handle the term  $ I_1,$ we choose a $C^{\infty}$ nonnegative cut-off function as the following
\begin{eqnarray*}
 \varphi(t)=
  \begin{cases}
  0, & t\leq \Delta t\\
  1, &  2\Delta t \leq t \leq T- 2\Delta t\\
  0, & t \geq T-\Delta t,
  \end{cases}
\end{eqnarray*}
where $ \Delta t>0 $ is small and $|\varphi'(t)|\leq \frac{1}{\Delta t}$ .
Notice that $ \varphi(t) \in \mathcal{D}(0,+ \infty) $, and
 $ \varphi(t)$ converges to $1$ almost everywhere as $ \Delta t \rightarrow 0.$

  By using H\"{o}lder inequality, Corollary \ref{cor-2} and \eqref{u-conv}, as $\Delta t \rightarrow 0$,
  we have
 \begin{align}\label{I1}
   I_1=&\int _0^{T- \Delta t}\int _{\mathbb{R}^3} [\rho {\bf u}(t+\Delta t,{\bf x})
  -\rho {\bf u}(t,{\bf x})]~ \big [\big(~{\bf u}(t+\Delta t,{\bf x})
  - \varphi(t){\bf u}^\epsilon(t+\Delta t,{\bf x})~\big)
  -\big(~{\bf u}(t,{\bf x})-\varphi(t){\bf u}^ \epsilon(t,{\bf x})~\big)\big]\, \mathrm{d}{\bf x}\, \mathrm{d}t \nonumber \\
  \leq&C \int_0^{T- \Delta t} \|\sqrt{\rho} {\bf u}\|_{L^q}\big[~\|{\bf u}_1- {\bf u}_1^\epsilon\|_{L^{l}}\|\sqrt{\rho}\|_{L^{p_1}}+\|(1- \varphi(t)){\bf u}_1^\epsilon\|_{L^{l}}\|\sqrt{\rho}\|_{L^{p_1}}\big]\, \mathrm{d}t \nonumber \\ 
  &+C \int_0^{T- \Delta t} \|\sqrt{\rho} {\bf u}\|_{L^2}\big[~\|{\bf u}_2- {\bf u}_2^\epsilon\|_{L^2}+\|(1- \varphi(t)){\bf u}_2^\epsilon\|_{L^2}\big]\, \mathrm{d}t \nonumber \\
  \leq& C\|\sqrt{\rho} {\bf u}\|_{L^q(0,T;L^q(\mathbb{R}^3))}\big[~\|{\bf u}_1- {\bf u}_1^\epsilon\|_{L^ {l}(0,T;L^{l}(\mathbb{R}^3))}+
   \|{\bf u}_1^\epsilon\|_{L^ {l}(0,T;L^{l}(\mathbb{R}^3))}\|1- \varphi\|_{L^{p_1}(0,T)}\big]
   \nonumber \\ 
   +& C\|\sqrt{\rho} {\bf u}\|_{L^\infty(0,T;L^2(\mathbb{R}^3))}\big[~\|{\bf u}_2- {\bf u}_2^\epsilon\|_{L^ {\infty}(0,T;L^2(\mathbb{R}^3))}+
   \|{\bf u}_2^\epsilon\|_{L^ \infty(0,T;L^2(\mathbb{R}^3))}\|1- \varphi\|_{L^\infty(0,T)}\big] \rightarrow 0,
  \end{align}
  where $ \frac{1}{q}+\frac{1}{l}+\frac{1}{p_1}=1, q>4.$

  Combining \eqref{I3},\eqref{I2} and \eqref{I1}, we complete the proof.
  \end{proof}

Thanks to Lemmas \ref{lem-R3}-\ref{lem-R4}, we have   $ L^2 $-equicontinuity of $ \sqrt{\rho}^\mu {\bf u}^\mu $. Thus, the following proposition follows.

\begin{proposition}\label{proposition1}
Under {\sl Assumption (RICKHw)},  for any $ T >0 $, there exists a subsequence (still denoted) $ \sqrt{\rho}^\mu {\bf u}^\mu $ and a function
$ \sqrt{\rho}{\bf u} \in L^2((0,T)\times \mathbb{R}^3) $ such that
   \begin{equation}\label{strong convergence}
     	\sqrt{\rho}^\mu {\bf u}^\mu \rightarrow \sqrt{\rho}{\bf u}~~~~~
      in~L^2((0,T)\times \mathbb{R}^3)\;~as~\mu\to0.
  \end{equation}

\end{proposition}

\section{Vanishing viscosity limit} In this section, we aim to prove our main result by the compactness argument as $\mu$ tends to zero.
Note that  $ \rho^\mu$ and $ \sqrt{\rho^{\mu}} $ are uniformly bounded in $ L^\infty([0,T]\times \mathbb{R}^3) \cap L^\infty(0,T;L^p(\mathbb{R}^3)),$ where $ 1 \leq p < \infty $, thus we have
\begin{align}\label{weak convergence of density}
  &\rho^{\mu}\to \rho  \quad\text{ weakly in } L^p([0,T]\times \mathbb{R}^3),~~1 < p<\infty, \nonumber \\
  &\sqrt{\rho^{\mu}}\to\sqrt{\rho}  \quad\text{ weakly in }  L^p([0,T]\times \mathbb{R}^3), ~~1 < p<\infty,
\end{align}
as $\mu\to 0.$
With the help of  Proposition \ref{proposition1}, for any test function $\boldsymbol{\phi} \in C_0^\infty(\mathbb{R}_+ \times \mathbb{R}^3;\mathbb{R}^3) $, we have
\begin{equation}
\begin{split}
\label{convergence of mass}
\int_0^T\int_{\mathbb{R}^3}\sqrt{\rho^{\mu}}\sqrt{\rho^{\mu}} {\bf u}^{\mu} \cdot
\boldsymbol{\phi}\,\mathrm{d}{\bf x}\,\mathrm{d}t\to \int_0^T\int_{\mathbb{R}^3}\sqrt{\rho}\sqrt{\rho}{\bf u}\cdot
\boldsymbol{\phi}\,\mathrm{d}{\bf x}\,\mathrm{d}t.
\end{split}
\end{equation}
Meanwhile, Proposition \ref{proposition1} yields that 
\begin{equation}
\begin{split}
\label{convergence of converction}
\int_0^T\int_{\mathbb{R}^3}\left(\sqrt{\rho^{\mu}}{\bf u}^{\mu}\otimes \sqrt{\rho^{\mu}} {\bf u}^{\mu}\right):\nabla\boldsymbol{\phi}\,\mathrm{d}{\bf x}\,\mathrm{d}t\to \int_0^T\int_{\mathbb{R}^3}\left(\sqrt{\rho}{\bf u}\otimes\sqrt{\rho}{\bf u}\right):\nabla\boldsymbol{\phi}\,\mathrm{d}{\bf x}\,\mathrm{d}t.
\end{split}
\end{equation}
The viscous term in the Navier-Stokes vanishes by letting $\mu$ tend to zero, in the following way
\begin{equation}\label{3}
 \Big|\int _0^T\int _{\mathbb{R}^3} \mu\nabla {\bf u}^\mu\cdot \nabla \boldsymbol{\phi}\, \mathrm{d}{\bf x}\, \mathrm{d}t \Big| \leq
 C \sqrt{\mu} ~\|\sqrt{\mu}\nabla {\bf u}^\mu\|_{L^2(0,T;L^2(\mathbb{R}^3))}\rightarrow 0.
\end{equation}

Note that in Definition \ref{defi-1}, the weak solutions $ (\rho^\mu, {\bf u}^\mu) $ of the Navier-Stokes equations satisfy the following weak formulation
\begin{align}\label{wf1}
	&\int _0^T\int _{\mathbb{R}^3} (\rho^\mu{\bf u}^\mu \cdot\boldsymbol{\phi}_t
  +(\rho^\mu {\bf u}^\mu \otimes {\bf u}^\mu): \nabla\boldsymbol{\phi}+\rho^\mu \mathbf{f}\cdot\boldsymbol{\phi})\, \mathrm{d}{\bf x}\, \mathrm{d}t +\int_{\mathbb{R}^3}{\bf m}_0({\bf x}) \cdot\boldsymbol{\phi}(0,{\bf x})\,\mathrm{d}{\bf x} \nonumber \\
	 =&\int _0^T\int _{\mathbb{R}^3} \mu\nabla {\bf u}^\mu\cdot \nabla \boldsymbol{\phi}\, \mathrm{d}{\bf x}\, \mathrm{d}t,
\end{align}
and
\begin{equation}\label{density equation}
  \int _0^T\int _{\mathbb{R}^3} \rho^\mu \phi_t+\rho^\mu {\bf u}^\mu \cdot \nabla \phi\, \mathrm{d}{\bf x}\, \mathrm{d}t+ \int_{\mathbb{R}^3} \rho_0^\mu \phi(0,{\bf x})\, \mathrm{d}{\bf x}=0.
\end{equation}

With  \eqref {weak convergence of density}-\eqref{3} at hand,  letting $\mu\to 0$ in  \eqref{wf1} and \eqref{density equation}, we have 
\begin{align}\label{momentum}
	\int _0^T\int _{\mathbb{R}^3} (\sqrt{\rho}\sqrt{\rho} {\bf u} \cdot\boldsymbol{\phi}_t
  +(\sqrt{\rho} {\bf u} \otimes \sqrt{\rho}{\bf u}): \nabla\boldsymbol{\phi}
  +\rho \mathbf{f} \cdot\boldsymbol{\phi})\, \mathrm{d}{\bf x}\, \mathrm{d}t
  +\int_{\mathbb{R}^3}{\bf m}_0({\bf x})\cdot\boldsymbol{\phi}(0,{\bf x})\,\mathrm{d}{\bf x}
	 =0,
\end{align}
and

\begin{equation}\label{mass}
\int _0^T\int _{\mathbb{R}^3} \rho \phi_t+\sqrt{\rho}\sqrt{\rho} {\bf u} \cdot \nabla \phi\, \mathrm{d}{\bf x}\, \mathrm{d}t+ \int_{\mathbb{R}^3} \rho_0 \phi(0,{\bf x})\, \mathrm{d}{\bf x}=0.	
\end{equation}

Since $ \int _0^t\int _{\mathbb{R}^3} \mu|\nabla {\bf u^\mu}|^2\, \mathrm{d}{\bf x}\, \mathrm{d} t \geq0 $, from \eqref{energy} we obtain
\begin{equation}
\label{energy for NS}
    \int _{\mathbb{R}^3} \frac{\rho^\mu |{\bf u^\mu}|^2}{2}\, \mathrm{d}{\bf x}
     \leq \int _{\mathbb{R}^3} \frac{|{\bf m}_0|^2}{2 \rho_0}\, \mathrm{d}{\bf x} +\int _0^t\int _{\mathbb{R}^3} \rho^\mu{\bf u^\mu} \cdot {\bf f}\, \mathrm{d}{\bf x}\, \mathrm{d} t.
\end{equation}
By   Proposition \ref{proposition1} and \eqref{weak convergence of density},  we have 
$$\int _0^t\int _{\mathbb{R}^3} \rho^\mu{\bf u^\mu} \cdot {\bf f}\, \mathrm{d}{\bf x}\, \mathrm{d} t\to \int _0^t\int _{\mathbb{R}^3} \rho{\bf u} \cdot {\bf f}\, \mathrm{d}{\bf x}\, \mathrm{d} t. $$
Thus,  by letting $\mu\to0$ in \eqref{energy for NS}, we deduce the following energy inequality
\begin{equation}\label{EE}
  \int _{\mathbb{R}^3} \frac{\rho |\bf u|^2}{2}\, \mathrm{d}{\bf x}
   \leq \int _{\mathbb{R}^3} \frac{|{\bf m}_0|^2}{2 \rho_0}\, \mathrm{d}{\bf x} +\int _0^t\int _{\mathbb{R}^3} \rho{\bf u} \cdot {\bf f}\, \mathrm{d}{\bf x}\, \mathrm{d} t.
\end{equation}
It is clear that \eqref{mass}, \eqref{momentum} and \eqref{EE} meet with  Definition \ref{EWD}. Therefore we conclude that $ (\rho, {\bf u}) $ is a weak solution of \eqref{Eulereq} with initial data
$ (\rho_0,{\bf m}_0)$.

\section{Appendix}\label{sec:appendix} 
In this section, we sketch the proof of Theorem \ref{rem:2}. In  the domain with period $ \mathbb{T}_P $, $ {\bf k}=(k_1,k_2,k_3)=\frac{2 \pi}{P}(n_1,n_2,n_3) \in \mathbb{R}^3, $ with $ n_j=0, \pm1, \pm2,\cdots, $ and $ j=1,2,3, $ is the discrete wavevector. Thus the total energy $ \mathcal{E}(t) $ per unit mass at time t for the inhomogeneous turbulence in $ \mathbb{T}_P $ is:
\begin{equation}
  \mathcal{E}(t)=\frac{1}{\int _{\mathbb{T}_P}\rho_0 \mathrm{d}{\bf x}}\int _{\mathbb{T}_P} \frac{\rho^\mu |{\bf u^\mu}|^2}{2}\, \mathrm{d}{\bf x}=\sum_{k \geq 0}E(t,k)= \sum_{k \geq 0}4 \pi q(t,k)k^2.
\end{equation}
And the Fourier transform of the weighted velocity $\sqrt{\rho}{\bf u}(t,{\bf x}) $ in the {\bf x}-variable is
\begin{equation*}
\widehat{\sqrt{\rho} {\bf u}}(t,\mathbf{k}) = \frac{1}{|\mathbb{T}_P|}\int _{\mathbb{T}_P} \sqrt{\rho} {\bf u}(t,{\bf x})e^{-i{\bf k} \cdot {\bf x}}\, \mathrm{d}{\bf x},
\end{equation*}
then we have
\begin{equation*}
  \sqrt{\rho}{\bf u}(t,{\bf x})=\sum_{\bf k} \widehat{\sqrt{\rho}{\bf u}}(t,{\bf k})e^{i {\bf k} \cdot {\bf x}} .
\end{equation*}

We have the following weaker version of {\sl Assumption (ICKHw)} in $\mathbb{T}_P $:
 \begin{assw}\label{ickhw}
For any $ T>0, $ there exists $ C_T>0 $ and $ k_{\ast} $(sufficiently large)  depending on $\rho_0,{\bf m}_0, P $ and $ \mathbf{f} $ but independent of the viscosity $\mu $ such that, for $ k=|\mathbf{k}|\geq k_ {\ast},$
\begin{equation}\label{ass-1}
	\sup_{k \geq k_{\ast}} \Big(|\mathbf{k}|^{3+ \beta}\int_0^T|\widehat{\sqrt{\rho}
	{\bf u}}(t,\mathbf{k})|^2\mathrm{d}t \Big)\leq C_T,~~for~some~\beta>0.
\end{equation}
 \end{assw}

By the {\sl Assumption (ICKHw)} and the Fourier transform, we can get the following uniform bound of $ \sqrt{\rho^\mu}{\bf u^\mu} $ in $ L^2(0,T;H^\alpha(\mathbb{T}_P)) $ for $ \alpha \in(0,1+\frac{\beta}{2}). $   For completeness, we present the proof which
is similar to \cite{Chen-Glimm2012,Chen-Glimm2019}.

 \begin{lemma}\label{lem-1}Under Assumption (ICKHw), for any $ T \in (0, +\infty), $ there exists $C>0,$
  independent of $ \mu>0, $ such that
 ~ \begin{equation}
 	\int _0^T\int _{\mathbb{T}_P} |D_x^{\alpha}( \sqrt{\rho^\mu} {\bf u}^\mu)|^2\, \mathrm{d}{\bf x} \, \mathrm{d} t \leq C,
\end{equation}
where $ \alpha \in(0, 1+\frac{\beta}{2}). $
 \end{lemma}
\begin{proof} Note that the definition of fractional derivatives via the Fourier transform, the Parseval identity and {\sl Assumption (ICKHw)} imply that
\begin{align}
	\int _0^T\int _{\mathbb{T}_P} |D_x^{\alpha}( \sqrt{\rho^\mu} {\bf u}^\mu)|^2\, \mathrm{d}{\bf x}\, \mathrm{d} t
	\leq&C\int _0^T \sum_{\bf k} |\widehat{ D_x^\alpha( \sqrt{\rho^\mu} {\bf u}^\mu)}|^2\, \mathrm{d}t \nonumber \\
	\leq& C\int _0^T \sum_{\bf k} |{\bf k}|^{2 \alpha}|\widehat{\sqrt{\rho^\mu}{\bf u}^\mu}|^2\, \mathrm{d}t \nonumber \\
	\leq& C\int _0^T \sum_{0 \leq |{\bf k}| \leq k_{\ast}}|{\bf k}|^{2 \alpha}|\widehat{\sqrt{\rho^\mu}{\bf u}^\mu}|^2 \, \mathrm{d}t
	+C\int _0^T \sum_{|{\bf k}| \geq k_{\ast}}|{\bf k}|^{2 \alpha}|\widehat{\sqrt{\rho^\mu}{\bf u}^\mu}|^2 \, \mathrm{d}t \nonumber \\
	 \leq&	C k_ \ast^{2 \alpha}\int _0^T\int _{\mathbb{T}_P} |\sqrt{\rho^\mu} {\bf u}^\mu|^2\, \mathrm{d}{\bf x}\, \mathrm{d}t+ C\sum_{|{\bf k}| \geq k_{\ast}} |{\bf k}|^{2 \alpha-3- \beta} \nonumber \\
	 \leq& C,
\end{align}	
where $ \alpha<1+\frac{\beta}{2}. $
\end{proof}
With Lemma \ref{lem-1} at hand, we are able to deduce the following high integrability of $ \sqrt{\rho^\mu}{\bf u}^\mu $, i.e., 
$$\|\sqrt{\rho^\mu}{\bf u}^\mu\|_{ L^q((0,T)\times \mathbb{T}_P)}\leq C,
~~for~q=q(\beta,r)>2.$$
It also gives us  the $L^2$-equicontinuity of $ \sqrt{\rho^\mu}{\bf u}^\mu $ in space variable {\bf x} in $ L^2((0,T)\times \mathbb{T}_P), $ independent of $ \mu. $

 \begin{lemma}\label{lem-3}Under {\sl Assumption (ICKHw)}, for any $ T \in (0, +\infty), $
 we get the equicontinuity of $ \sqrt{\rho^\mu} {\bf u}^\mu $ with respect to the space variable ${\bf x}$ in $ L^2((0,T)\times \mathbb{T}_P) $, independent of $ \mu, $ i.e.,
 ~\begin{equation}
 	\int _0^T\int _{\mathbb{T}_P} | \sqrt{\rho^\mu}{\bf u}^\mu(t,{\bf x}+ \Delta {\bf x})-\sqrt{\rho^\mu} {\bf u}^\mu(t,{\bf x})|^2\, \mathrm{d}{\bf x}\, \mathrm{d}t
 	\rightarrow 0,~~as ~ \Delta {\bf x} \rightarrow 0.
 \end{equation}
 \end{lemma}
 \begin{proof}Using Lemma \ref{lem-1} and Parseval identity, we obtain
 \begin{align*}
   \int _0^T\int _{\mathbb{T}_P} | \sqrt{\rho^\mu}{\bf u}^\mu(t,{\bf x}+ \Delta {\bf x})-\sqrt{\rho^\mu} {\bf u}^\mu(t,{\bf x})|^2\, \mathrm{d}{\bf x}\, \mathrm{d}t
   =& \int_0^T \sum_{\bf k} | \widehat{\sqrt{\rho^\mu}{\bf u}^\mu}|^2
   (e^{i{\bf k} \cdot \Delta {\bf x}}-1)^2\, d\mathrm{t} \nonumber \\
    \leq& C |\Delta {\bf x}|^{2 \alpha}\int_0^T \sum_{\bf k} |{\bf k}|^{2 \alpha}|\widehat{ \sqrt{\rho^\mu}{\bf u}^\mu}|^2 \, d\mathrm{t} \nonumber \\
   \leq& C |\Delta {\bf x}|^{2 \alpha} \int _0^T\int _{\mathbb{T}_P} |D_{\bf x}^\alpha(\sqrt{\rho^\mu}{\bf u}^\mu)|^2\, \mathrm{d}{\bf x}\, \mathrm{d}t  \nonumber \\
    \leq& C|\Delta {\bf x}|^{2 \alpha} .
 \end{align*}
 \end{proof}
For $ L^2- $ equicontinuity of $ \sqrt{\rho^\mu}{\bf u}^\mu $  with respect to $t$, the proof is similar to Lemma \ref{lem-R4} with some slight modification. Since the domain considered is a domain with period $\mathbb{T}_P=[-P/2,P/2]^3 \subset \mathbb{R}^3, P>0.$ Hence, it need not to estimate $ I_i~(i=1,2,3,4)$ by dividing $ {\bf u}$ into ${\bf u}_1 $ and $ {\bf u}_2$ .  We here only state the following lemma.
\begin{lemma}\label{lem-4}Under {\sl Assumption (ICKHw)}, for any $ T >0 $,
 we have the equicontinuity of $ \sqrt{\rho^\mu} {\bf u}^\mu $ with respect to the time variable t in $ L^2((0,T- \Delta t )\times \mathbb{T}_P) $, independent of $ \mu,$ i.e.
 ~\begin{equation}
 	\int _0^{T-\Delta t}\int _{\mathbb{T}_P}
 	| \sqrt{\rho^\mu}{\bf u}^\mu(t+ \Delta t,{\bf x})
 	-\sqrt{\rho^\mu} {\bf u}^\mu(t,{\bf x})|^2\, \mathrm{d}{\bf x}\, \mathrm{d}t
 	\rightarrow 0,~~as ~ \Delta t \rightarrow 0.
 \end{equation}
 \end{lemma}
 With the $ L^2- $ equicontinuity of $ \sqrt{\rho^\mu}{\bf u}^\mu $  with respect to $x$ and $ t $, we directly deduce that there exists a subsequence (still denoted) $ \sqrt{\rho}^\mu {\bf u}^\mu $ and a function
$ \sqrt{\rho}{\bf u} \in L^2((0,T)\times \mathbb{T}_P) $ such that
$$ \sqrt{\rho}^\mu {\bf u}^\mu \rightarrow \sqrt{\rho}{\bf u}~~~~~
      in~L^2((0,T)\times \mathbb{T}_P)~as~\mu\to0. $$
Finally we can get a same theorem as Theorem \ref{thm-R1} as follows: 
\begin{theorem}\label{thm-1} 
Under Assumption (ICKHw) \eqref{ass-1}, for the weak solution $ (\rho^\mu, {\bf u^\mu}) $ of \eqref{maineq}-\eqref{I.C.} as in Definition \ref{defi-1}, there exists a subsequence (still denote) $ (\rho^\mu, {\bf u^\mu}) $ and a function $ (\rho, {\bf u}) $ such that as $ \mu \rightarrow 0,$
\begin{equation}\label{conv}
 \rho^\mu \rightarrow \rho ~~~weakly~in~L^p((0,T)\times \mathbb{T}_P),~
 \sqrt{\rho^\mu}{\bf u^\mu} \rightarrow \sqrt{\rho}{\bf u},~in~L^2((0,T)\times \mathbb{T}_P),
\end{equation}
 where $ 1 < p<\infty $, and $ (\rho, {\bf u}) $ is a weak solution of \eqref{Eulereq} with the initial data $ (\rho_0,{\bf m}_0) $.
\end{theorem}
\vspace{5mm}
\section*{Acknowledgments} 

Cheng Yu is partially supported by Collaboration Grants for Mathematicians from Simons Foundation.


\addcontentsline{toc}{section}{References}

\end{document}